\documentclass[reqno,12pt]{amsart}
\usepackage[T1]{fontenc}
\usepackage[utf8]{inputenc}
\usepackage[english]{babel}
\usepackage{amssymb,amsmath,amsthm,amsfonts,xcolor,enumerate,hyperref,comment,longtable,cleveref,mathtools}

\usepackage{times}
\usepackage{cite}
\usepackage{pdflscape}
\usepackage[mathcal]{euscript}
\usepackage{tikz}
\usepackage{cancel}
\usepackage{stmaryrd}
\usepackage{longtable}

\usepackage[a4paper,top=3cm, bottom=3cm, left=3cm, right=3cm]{geometry}

\theoremstyle{plain}
\newtheorem{thm}[subsection]{Theorem}
\newtheorem{prop}[subsection]{Proposition}
\newtheorem{cor}[subsection]{Corollary}	
	
\theoremstyle{definition}
\newtheorem{defn}[subsection]{Definition}
\newtheorem{exm}[subsection]{Example}
\newtheorem{rem}[subsection]{Remark}

\begin{document}


\title[Local automorphisms  of $p$-filiform Leibniz  algebras]{Local automorphisms  of $p$-filiform Leibniz  algebras}

\author[ Yusupov B.B. ]{Bakhtiyor Yusupov$^{1,2}$}
\address{$^1$ V.I.Romanovskiy Institute of Mathematics\\
  Uzbekistan Academy of Sciences, 9 \\ Univesity street, 100174  \\
  Tashkent,   Uzbekistan}
\address{$^2$ Department of Physics and Mathematics, Urgench State University, H. Alimdjan street, 14, Urgench
220100, Uzbekistan}
\email{\textcolor[rgb]{0.00,0.00,0.84}{baxtiyor\_yusupov\_93@mail.ru}}
\maketitle

\date{}
\maketitle

\begin{abstract}
This paper is devoted to study local automorphisms  of $p$-filiform Leibniz algebras. We prove that $p$-filiform Leibniz algebras as a rule admit local automorphisms which are not automorphisms.\\

{\it Keywords:} Leibniz algebra, $p$-filiform Leibniz algebras, automorphism, local automorphism.
\\

{\it AMS Subject Classification:} 17A36,  17B20, 17B40.

\end{abstract}

\maketitle \thispagestyle{empty}


\section{Introduction}\label{sec:intro}

In recent years non-associative analogues of classical constructions become of
interest in connection with their applications in many branches of mathematics and
physics. The notions of local and $2$-local derivations (automorphisms) are also  popular for some non-associative algebras such as the Lie and Leibniz algebras.

R.Kadison \cite{Kad} introduced the concept of a local derivation and proved
that each continuous local derivation from a von Neumann algebra into its dual Banach
bemodule is a derivation. B. Jonson \cite{Jon} extended the above result by proving that every
local derivation from a C*-algebra into its Banach bimodule is a derivation. In particular, Johnson
gave an automatic continuity result by proving that local derivations of a C*-algebra $A$ into a
Banach $A$-bimodule $X$ are continuous even if not assumed a priori to be so
(cf. \cite[Theorem 7.5]{Jon}). Based on these results, many authors have studied
local derivations on operator algebras.

A similar notion, which characterizes non-linear generalizations of automorphisms, was introduced
by  P.\v{S}emrl in \cite{S} as $2$-local automorphisms.
He described such maps on the algebra $B(H)$ of all bounded linear operators on an infinite dimensional separable Hilbert space $H$.
After P.\v{S}emrl's work, numerous new results related to the description of local and $2$-local derivation of some varieties have been appeared
(see, for example, \cite{AKO,WCN1,ChWD}).

The first results concerning to local and 2-local derivations and automorphisms on finite-dimensional Lie algebras over algebraically closed
field of zero characteristic were obtained in \cite{AK,AKR}. Namely, in \cite{AKR} it is proved
that every 2-local derivation on a semi-simple Lie algebra $\mathcal{L}$ is a derivation and that
each finite-dimensional nilpotent Lie algebra with dimension larger than two admits
2-local derivation which is not a derivation. In \cite{AK} the authors have proved that every local
derivation on semi-simple Lie algebras is a derivation and gave examples of nilpotent
finite-dimensional Lie algebras with local derivations which are not derivations.
 Sh.Ayupov, K.Kudaybergenov, B.Omirov proved similar results concerning local and 2-local derivations and automorphisms on simple Leibniz algebras  in their recent paper \cite{AKO}. Local automorphisms of certain finite-dimensional simple Lie and Leibniz algebras are investigated in \cite{AKConfe}. Concerning local automorphism, T.Becker, J.Escobar, C.Salas and R.Turdibaev in \cite{BEST} established that the set of local automorphisms $LAut(sl_2)$  coincides with the group $Aut^{\pm}(sl_2)$ of all automorphisms and anti-automorphisms. Later in \cite{Costantini}  M.Costantini proved that a linear map on a simple Lie algebra is a local automorphism if
and only if it is either an automorphism or an anti-automorphism. Similar results concerning local and 2-local derivations and automorphisms on Lie superalgebras were obtained in  \cite{ChWD,WCN1} and \cite{WCN2}.
 In \cite{AyuKhuYus} local derivations of solvable Leibniz algebras are investigated and it is shown that in the class of solvable Leibniz algebras there exist algebras which admit local derivations which are not derivation and also algebras for which every local derivation
is a derivation. Moreover, it is proved that every local derivation on a finite-dimensional solvable Leibniz algebras with model
nilradical and maximal dimension of complementary space is a derivation. The results of the paper \cite{AyuKudYus1} show that p-filiform Leibniz algebras as a rule admit local derivations which are not derivations. The authors proved similar results concerning local automorphism on the solvable Leibniz algebras with null-filiform and naturally graded non-Lie filiform nilradicals, whose dimension of complementary space is maximal is an automorphism\cite{AKU}. J.Adashev and B.Yusupov proved similar results concerning local derivations of naturally graded quasi-filiform Leibniz algebras in their recent paper \cite{AdaYus}. J.Adashev and B.Yusupov proved proved that direct sum null-filiform nilpotent Leibniz algebras as a rule admit local derivations which are not derivations \cite{AdaYus1}. J.Adashev and B.Yusupov proved that quasi-filiform Leibniz algebras of type I, as a rule, admit local automorphisms which are not automorphisms \cite{AdaYus2}.
The first example of a simple (ternary) algebra with nontrivial local derivations is constructed by B.Ferreira, I.Kaygorodov and K.Kudaybergenov in \cite{FKK}. After that, in \cite{AEK1} Sh.Ayupov, A.Elduque and K.Kudaybergenov constructed an example for a simple (binary) algebra with non-trivial local derivations.

In the paper \cite{KUY}, I.A.Karimjanov, S.M.Umrzaqov, and B.B.Yusupov describe automorphisms, local and 2-local automorphisms of solvable Leibniz algebras with a model or abelian null-radicals. They show that any local automorphisms on solvable Leibniz algebras with a model nilradical, the dimension of the complementary space of which is maximal, is an automorphism.
But solvable Leibniz algebras with an abelian nilradical with a  $1$-dimensional complementary space admit local automorphisms which are not automorphisms.

In the present paper we study automorphisms and local automorphisms $p$-filiform Leibniz algebras. In Section 3 we describe the automorphisms $p$-filiform Leibniz algebras. In Section 4 we describe the local automorphisms $p$-filiform Leibniz algebras. We show that in section 4 we describe $p$-filiform Leibniz algebras as a rule admit local automorphisms which are not automorphisms.

\section{Preliminaries}

In this section we give some necessary definitions and preliminary results.

\begin{defn}
A  vector space with bilinear bracket $(\mathcal{L},[\cdot,\cdot])$
is called a Leibniz algebra if for any $x,y,z\in \mathcal{L}$ the so-called
Leibniz identity
$$\big[x,[y,z]\big]=\big[[x,y],z\big]-\big[[x,z],y\big],$$
holds.
\end{defn}

Let $\mathcal{L}$ be a Leibniz algebra. For a Leibniz algebra $\mathcal{L}$ consider the following central lower and
derived sequences:
$$
\mathcal{L}^1=\mathcal{L},\quad \mathcal{L}^{k+1}=[\mathcal{L}^k,\mathcal{L}^1], \quad k \geq 1,
$$
$$\mathcal{L}^{[1]} = \mathcal{L}, \quad \mathcal{L}^{[s+1]} = [\mathcal{L}^{[s]}, \mathcal{L}^{[s]}], \quad s \geq 1.$$

\begin{defn} A Leibniz algebra $\mathcal{L}$ is called
nilpotent (respectively, solvable), if there exists  $p\in\mathbb N$ $(q\in
\mathbb N)$ such that $\mathcal{L}^p=0$ (respectively, $\mathcal{L}^{[q]}=0$).The minimal number $p$ (respectively, $q$) with such
property is said to be the index of nilpotency (respectively, of solvability) of the algebra $\mathcal{L}$.
\end{defn}
Now let us define a naturally graduation for a nilpotent Leibniz algebra.

\begin{defn} Given a nilpotent Leibniz algebra $\mathcal{L}$, put
$\mathcal{L}_i=\mathcal{L}^i/\mathcal{L}^{i+1}, \ 1 \leq i\leq n-1$, and $gr(\mathcal{L}) = \mathcal{L}_1 \oplus
\mathcal{L}_2\oplus\dots \oplus \mathcal{L}_{n-1}$. Then $[\mathcal{L}_i,\mathcal{L}_j]\subseteq \mathcal{L}_{i+j}$ and we
obtain the graded algebra $gr(\mathcal{L})$. If $gr(\mathcal{L})$ and $L$ are isomorphic, then
we say that an algebra $\mathcal{L}$ is naturally graded.
\end{defn}

Now we define the notion of characteristic sequence, which is one of the important invariants.
For a finite-dimensional nilpotent Leibniz algebra  $N$ and for the matrix of the linear operator $R_x$ denote by $C(x)$ the
descending sequence of its Jordan blocks' dimensions. Consider the
lexicographical order on the set $C(N)=\{C(x)  \ | \ x \in N\}$.

\begin{defn} The sequence
$$\left(\max\limits_{x\in N\setminus N^2} C(x) \right) $$ is said to be the
characteristic sequence of the nilpotent Leibniz algebra $N.$
\end{defn}

\begin{defn} A Leibniz algebra $\mathcal{L}$ is called $p$-filiform, if the characteristic sequence is $C(\mathcal{L})=(n-p,\underbrace {1,\dots,1}_{p}).$
\end{defn}

Now we give the definitions of automorphisms and local automorphisms.

\begin{defn}\label{aut}
A linear bijective map $\varphi: \mathcal{L} \rightarrow \mathcal{L}$ is called an automorphism, if it satisfies $\varphi([x,y])=[\varphi(x),\varphi(y)]$  for all $x,y\in\mathcal{L}.$
\end{defn}

\begin{defn}
Let $\mathcal{L}$ be an algebra. A linear map $\Delta : \mathcal{L} \to \mathcal{L}$ is called a local automorphism, if for
any element $x \in\mathcal{L} $ there exists an automorphism $\varphi_x : \mathcal{L} \to \mathcal{L}$ such that $\Delta(x) = \varphi_x(x)$.
\end{defn}

It was proved in \cite[Theorem 2.9]{ALO} that any naturally graded indecomposable non-Lie $p$-filiform Leibniz algebra, is isomorphic to one of the following pairwise non-isomorphic algebras $(n-p\geq4)$:

if $p=2k$, then
\begin{eqnarray*}
& \mu_1  : & \left\{\begin{array}{ll}
[e_i,e_1]=e_{i+1},    & 1\leq i\leq n-2k-1,\\[1mm]
[e_1, f_j] =f_{k+j},  & 1\leq j\leq k,\\[1mm]
 \end{array}\right.
\\
&\mu_2 : &\left\{\begin{array}{ll}
[e_i,e_1]=e_{i+1},       & 1\leq i\leq n-2k-1,\\[1mm]
[e_1,f_1]=e_{2}+f_{k+1}, &\\[1mm]
[e_i,f_1]=e_{i+1},       & 2\leq i\leq n-2k-1,\\[1mm]
[e_1, f_j] =f_{k+j},     & 2\leq j\leq k,\\[1mm]
 \end{array}\right.
\end{eqnarray*}

if $p=2k+1$, then
\begin{eqnarray*}
&\mu_3 : & \left\{\begin{array}{ll}
[e_1,e_1]=e_{3},   & \\[1mm]
[e_i,e_1]=e_{i+1}, & 2\leq i\leq n-2k-1,\\[1mm]
[e_1,f_j]=f_{k+j}, & 1\leq j\leq k,\\[1mm]
[e_2,f_j]=f_{k+j}, & 1\leq j\leq k,\\[1mm]
 \end{array}\right.
\end{eqnarray*}
where  $\{e_1,e_2,\dots,e_{n-p},f_1,f_2,\dots,f_{p}\}$ is the basis of the algebra and the omitted products are equal to zero.

\section{Automorphisms of naturally graded non-Lie $p$-filiform Leibniz  algebras}

In order to start the description we need to know the automorphisms of naturally graded non-Lie $p$-filiform Leibniz algebras.

\begin{prop} \label{prop1} Any automorphisms of the algebra $\mu_1$ has the following matrix form:
\begin{eqnarray*}
& \Phi & =\begin{pmatrix}
\varphi_{1,1} & \varphi_{1,2}\\
\varphi_{2,1} & \varphi_{2,2}
\end{pmatrix},
\end{eqnarray*}
where
$$
\varphi_{1,1} = \sum_{i=1}^{n-2k}a_{1}^ie_{i,i}+\sum_{i=1}^{n-2k-1}a_1^{i-1}\sum_{j=i+1}^{n-2k}a_{j-i+1}e_{j,i}, $$
$$ \varphi_{2,1}=\sum_{i=1}^{2k}b_{i}e_{i,1}+a_1\sum_{i=1}^{k}b_{i}e_{k+i,2},\quad
\varphi_{1,2} = \sum_{i=1}^{k}c_ie_{n-2k,i},
$$

$$ \varphi_{2,2}= \begin{pmatrix}
\varphi_{2,2}^{(1)} & 0\\
\varphi_{2,2}^{(2)} & a_1\varphi_{2,2}^{(1)}
\end{pmatrix},$$

$$\varphi_{1,1}\in M_{n-2k,n-2k}, \ \varphi_{2,1}\in M_{2k,n-2k}, \
\varphi_{1,2}\in M_{n-2k,2k}, \ \varphi_{2,2}^{(1)}, \varphi_{2,2}^{(2)}.$$

Let $\{e_{i,j}: 1 \leq i,j\leq n\}$ -- be the system of matrix units, i.e., the $(n\times n)$-matrix $e_{i,j}$ is such that the
$(i,j)$th component is equal to 1 and all other components are zeros.
\end{prop}
\begin{proof} Let $\{e_1,f_1,f_2,\dots,f_k\}$ be a generator basis elements of the algebra $\mu_1$.

We put
\[ \varphi(e_1)=\sum\limits_{i=1}^{n-2k}a_ie_i+\sum\limits_{i=1}^{2k}b_if_i, \qquad \varphi(f_i)=\sum\limits_{j=1}^{n-2k}c_{j,i}e_j+\sum\limits_{j=1}^{2k}d_{j,i}f_j, \quad 1\leq i\leq k.\]

From the automorphisms property \eqref{aut} we have
\begin{equation*}\begin{split}
\varphi(e_2)&=\varphi([e_1,e_1])=[\varphi(e_1),\varphi(e_1)]=
\left[\sum\limits_{i=1}^{n-2k}a_ie_i+\sum\limits_{i=1}^{2k}b_if_i,\sum\limits_{i=1}^{n-2k}a_ie_i+\sum\limits_{i=1}^{2k}b_if_i\right]=\\
            &=a_1\sum\limits_{i=2}^{n-2k}a_{i-1}e_{i}+a_1\sum\limits_{i=1}^{k}b_if_{k+i}.
\end{split}
\end{equation*}

Buy applying the induction and the automorphisms property \eqref{aut} we derive
\[ \varphi(e_i)=a_1^{i-1}\sum\limits_{t=i}^{n-2k}a_{t-i+1}e_t, \quad 3 \leq i \leq
n-2k.\]

Consider
\begin{equation*}\begin{split}
0&=\varphi([f_i,e_1])=[\varphi(f_i),\varphi(e_1)]=
\left[\sum\limits_{j=1}^{n-2k}c_{j,i}e_j+\sum\limits_{j=1}^{2k}d_{j,i}f_j,\sum\limits_{j=1}^{n-2k}a_je_j+\sum\limits_{j=1}^{2k}b_jf_j\right]=\\
&=a_1\sum\limits_{j=1}^{n-2k-1}c_{j,i}e_{j+1}+c_{1,i}\sum\limits_{j=1}^kb_jf_j, \quad 1 \leq i \leq k.
\end{split}
\end{equation*}
Consequently,
\[a_1\neq0,\ c_{j,i}=0,\quad  1 \leq i \leq k, \quad 1 \leq j \leq n-2k-1.\]

Similarly, from $\varphi(f_{k+i})=\varphi([e_1, f_i]), \ 1 \leq i \leq k$, we deduce
\[ \varphi(f_{k+i})=a_1\sum\limits_{j=1}^{k}d_{j,i}f_{k+j}, \qquad 1\leq i\leq k.\]

\end{proof}

\begin{prop} \label{prop2} Any automorphisms of the algebra $\mu_2$ has the following matrix form:

\begin{eqnarray*}
& \Phi & =\begin{pmatrix}
\varphi_{1,1} & \varphi_{1,2}\\
\varphi_{2,1} & \varphi_{2,2}
\end{pmatrix},
\end{eqnarray*}
where
$$ \varphi_{1,1}=  \sum_{i=1}^{n-2k}(a_{1}+b_1)^{i-1}a_1e_{i,i}+\sum_{i=1}^{n-2k-1}\sum_{j=i+1}^{n-2k}(a_1+b_1)^{i-1}a_{j-i+1}e_{j,i}, $$
$$\varphi_{2,1} =\sum_{i=1}^{2k}b_{i}e_{i,1}+a_1\sum_{i=1}^{k}b_{i}e_{k+i,2},\quad D_{1,2}= \sum_{i=1}^{k}c_ie_{n-2k,i},$$
$$ \varphi_{2,2}=  \begin{pmatrix}
\varphi_{2,2}^{(1)} & 0\\
\varphi_{2,2}^{(2)} & \varphi_{2,2}^{(3)}
\end{pmatrix},$$
$$ \varphi_{2,2}^{(1)} =  \sum_{i=1}^{k}\sum_{j=2}^{k}d_{j,i}e_{j,i}+(a_1+b_{1})e_{1,1},\quad
\varphi_{2,2}^{(3)} =  a_1\varphi_{2,2}^{(1)}-a_1\sum_{j=1}^{k}b_{j}e_{j,1},
$$
$$\varphi_{1,1}\in M_{n-2k,n-2k},\  \varphi_{2,1}\in M_{2k,n-2k}, \ \varphi_{1,2}\in M_{n-2k,2k},\
\varphi_{2,2}^{(1)}, \varphi_{2,2}^{(2)}, \varphi_{2,2}^{(3)}, \mathbb{E}\in M_{k,k}.$$
Let $\{e_{i,j}: 1 \leq i,j\leq n\}$ -- be the system of matrix units, i.e., the $(n\times n)$-matrix $e_{i,j}$ is such that the
$(i,j)$th component is equal to 1 and all other components are zeros.
\end{prop}

\begin{proof}The proof follows by straightforward calculations similarly to the proof of Proposition \ref{prop1}.

\end{proof}

\begin{prop} \label{prop3} Any automorphisms of the algebra $\mu_3$ has the following matrix form:

\begin{eqnarray*}
& \Phi & =\begin{pmatrix}
\varphi_{1,1} & \varphi_{1,2}\\
\varphi_{2,1} & \varphi_{2,2}
\end{pmatrix},
\end{eqnarray*}
where
\begin{eqnarray*}
& \varphi_{1,1}= & a_1e_{1,1}+\sum_{i=2}^{n-2k}a_1^{i-2}(a_{1}+a_2)e_{i,i}+\sum_{i=2}^{n-2k}a_{i}e_{i,1}+\sum_{i=3}^{n-2k-1}a_{i}e_{i,2}+\\
&&+\beta e_{n-2k,2}+\sum_{i=3}^{n-2k-1}\sum_{j=i+1}^{n-2k}a_1^{j-3}a_{j-i+2}e_{j,i},\\
&\varphi_{2,1}= & \sum_{i=1}^{2k}b_{i,1}e_{i,1}+\sum_{i=1}^{k}b_{i,2}e_{k+i,2}+(a_1+a_2)\sum_{i=1}^{k}b_{i,1}e_{k+i,3},\quad  \varphi_{1,2}= \sum_{i=1}^{k}c_ie_{n-2k,i},\\
& \varphi_{2,2}= & \begin{pmatrix}
\varphi_{2,2}^{(1)} & 0\\
\varphi_{2,2}^{(2)} &(a_1+a_2)\varphi_{2,2}^{(1)}
\end{pmatrix}
\end{eqnarray*}
$$\varphi_{1,1}\in M_{n-2k,n-2k}, \ \varphi_{2,1}\in M_{2k,n-2k}, \ \varphi_{1,2}\in M_{n-2k,2k}, \ \varphi_{2,2}^{(1)},\varphi_{2,2}^{(2)}\in M_{k,k}.$$
Let $\{e_{i,j}: 1 \leq i,j\leq n\}$ -- be the system of matrix units, i.e., the $(n\times n)$-matrix $e_{i,j}$ is such that the
$(i,j)$th component is equal to 1 and all other components are zeros.
\end{prop}

\begin{proof}The proof follows by straightforward calculations similarly to the proof of Proposition \ref{prop1}.

\end{proof}

\begin{rem}\label{remsec 1.4.1}
The dimensions of the space of automorphisms of the algebras $\mu_{1}, \mu_{2}$ and $\mu_{3}$ are
\begin{eqnarray*}
& \dim Aut(\mu_{1})= & n+2k^2+k,    \\
& \dim Aut(\mu_{2})= & n+2k^2+1 ,    \\
& \dim Aut(\mu_{3})= & n+2k^2+2k+1,
\end{eqnarray*}
where $k\in \mathbb{N}$ and $n\geq 2k+4.$
\end{rem}

\section{Local automorphisms of naturally graded non-Lie $p$-filiform Leibniz  algebras}

In the following theorem we give the description of local automorphisms of the algebra $\mu_{1}$.
\begin{thm}\label{sec1.4.2}   Let $\Delta$ be a linear operator on $\mu_{1}$. Then $\Delta$ is a local automorphisms, if and only if its matrix has the form:
\begin{equation}\label{for1.4.1}
\Delta= \left(\begin{array}{cc}
\Delta_{1,1} & \Delta_{1,2}\\
\Delta_{2,1} & \Delta_{2,2}
\end{array}\right),
\end{equation}
where
$$
 \Delta_{1,1}= \sum\limits_{j=1}^{n-2k}\sum\limits_{i=j}^{n-2k}\gamma_{i,j}e_{i,j}, \quad     \Delta_{2,1}=\sum\limits_{i=n-2k+1}^{n}\gamma_{i,1}e_{i,1}+\sum\limits_{i=n-k+1}^{n}\gamma_{i+1,2}e_{i+1,2},$$
$$ \Delta_{1,2}= \sum\limits_{i=n-2k+1}^{n-k}\gamma_{n-2k,i}e_{n-2k,i},$$
$$ \Delta_{2,2}= \begin{pmatrix}
\Delta_{2,2}^{(1)} & 0\\
\Delta_{2,2}^{(2)} & \Delta_{2,2}^{(3)}
\end{pmatrix},$$
$$
 \Delta_{2,2}^{(1)} =  \sum\limits_{j=n-2k+1}^{n-k}\sum\limits_{j=n-2k+1}^{n-k}\gamma_{i,j}e_{i,j},\quad
\Delta_{2,2}^{(2)}=
\sum\limits_{j=n-2k+1}^{n-k}\sum\limits_{j=n-k+1}^{n}\gamma_{i,j}e_{i,j},$$
$$\Delta_{2,2}^{(3)} =
\sum\limits_{j=n-k+1}^{n}\sum\limits_{j=n-k+1}^{n}\gamma_{i,j}e_{i,j}.
$$
\end{thm}

\begin{proof}
 $(\Rightarrow)$ Assume that $\Delta$ is a local automorphisms  of $\mu_{1}:$
\begin{eqnarray*}
\Delta= \left(\begin{array}{cc}
\Delta_{1,1} & \Delta_{1,2}\\
\Delta_{2,1} & \Delta_{2,2}
\end{array}\right),
\end{eqnarray*}
where
\begin{eqnarray*}
& \Delta_{1,1}=& \sum\limits_{j=1}^{n-2k}\sum\limits_{i=1}^{n-2k}\gamma_{i,j}e_{i,j},\quad
\Delta_{2,1} = \sum\limits_{i=n-2k+1}^{n}\sum\limits_{j=1}^{n-2k}\gamma_{i,j}e_{i,j},\\
& \Delta_{1,2}=& \sum\limits_{j=n-2k+1}^{n}\sum\limits_{i=1}^{n-2k}\gamma_{i,j}e_{i,j},\quad
 \Delta_{2,2}=  \sum\limits_{j=n-2k+1}^{n}\sum\limits_{j=n-2k+1}^{n}\gamma_{i,j}e_{i,j}.
\end{eqnarray*}

Take a automorphism $\varphi_{e_{2}}$ such that $\Delta(e_{2})=\varphi_{e_{2}}(e_{2})$. Then
\begin{eqnarray*}
& \Delta(e_{2})   =& \sum_{j=1}^{n-2k}\gamma_{j,2}e_{2}+\sum_{j=n-2k+1}^{n}\gamma_{j,2}f_{j-n+2k},\\
& \varphi_{e_{2}}(e_{2})=& a_1\sum_{i=2}^{n-2k}a_{i-1}e_{i}+a_1\sum_{i=1}^kb_{i}f_{k+i}.
\end{eqnarray*}

Comparing the coefficients, we conclude that
$\gamma_{1,2}=\gamma_{n-2k+i,2}=0$ for $1\leq i\leq k$.

We take a automorphism $\varphi_{e_{i}}$ such that
$\Delta(e_{i})=\varphi_{e_{i}}(e_{i}),$ where $3\leq i\leq n-2k.$ Then
\begin{eqnarray*}
& \Delta(e_{i})   =& \sum_{j=1}^{n-2k}\gamma_{j,i}e_{j}+\sum_{j=n-2k+1}^{n}\gamma_{j,i}f_{j-n+2k},\\
& \varphi_{e_{i}}(e_{i})=& a_{1}^{i-1}\sum_{j=i}^{n-2k}a_{j-i+1}e_{j}.
\end{eqnarray*}

Comparing the coefficients at the basis elements for $\Delta(e_{i})$ and $\varphi_{e_{i}}(e_{i})$, we obtain the identities  $$\gamma_{t,j}=\gamma_{n-2k+i,j}=0,\quad 3\leq j\leq n-2k,\  1\leq i\leq 2k,\  1\leq t\leq n-2k-1.$$

We take a automorphism $\varphi_{f_{i}}$ such that $\Delta(f_{i})=\varphi_{f_{i}}(f_{i}),$ where $1\leq i\leq k$.
Then
\begin{eqnarray*}
& \Delta(f_{i})   =& \sum_{j=1}^{n-2k}\gamma_{j,n-2k+i}e_{j}+\sum_{j=n-2k+1}^{n}\gamma_{j,n-2k+i}f_{j-n+2k},\\
& \varphi_{f_{i}}(f_{i})=& c_{n-2k,i}e_{n-2k}+\sum_{j=1}^{2k}d_{j,i}f_{j}.
\end{eqnarray*}

Comparing the coefficients at the basis elements for  $\Delta(f_{i})$ and $\varphi_{f_{i}}(f_{i})$, we obtain
 $$\gamma_{j,i}=\gamma_{j,n-2k+i}=0,\quad 1\leq j\leq n-2k-1,\ 1\leq i\leq k-1.$$

Now, take a  automorphism  $\varphi_{f_{i}}$ such that $\Delta(f_{i})=\varphi_{f_{i}}(f_{i}),$ where  $k+1\leq i\leq 2k$. Then
\begin{eqnarray*}
& \Delta(f_{i})   =& \sum_{j=1}^{n-2k}\gamma_{j,n-2k+i}e_{j}+\sum_{j=n-2k+1}^{n}\gamma_{j,n-2k+i}f_{j-n+2k},\\
& \varphi_{f_{i}}(f_{i})=& a_1\sum_{j=1}^{k}d_{j,i-k}f_{k+j},
\end{eqnarray*}
which implies
$$\gamma_{j,i}=0, \quad 1\leq j\leq n-k, \ n-k+1\leq i\leq n.$$

 $(\Leftarrow)$ Assume that the operator $\Delta$ has the form (\ref{for1.4.1}).
For an arbitrary element
$$x=\sum\limits_{i=1}^{n-2k}
\xi_{i}e_{i}+\sum\limits_{i=1}^{2k}\zeta_{i}f_{i},$$
we have
\begin{eqnarray*}
&\varphi(x)_{e_1}   = & a_{1}\xi_{1},\\
& \varphi(x)_{e_i}   = & a_{i}\xi_{1}+\sum_{j=1}^{i-2}a^{j}_1a_{i-j}\xi_{j+1}+a_{1}^i\xi_{i}, \ \ \ \ 2\leq i\leq n-2k-1,\\
& \varphi(x)_{e_{n-2k}}= & a_{n-2k}\xi_{1}+ \sum_{j=1}^{n-2k-2}a_{n-2k-j}\xi_{j+1}+(n-2k)a_{1}\xi_{n-2k}+
\sum_{j=1}^kc_{j,n-2k}\zeta_{j},\\
& \varphi(x)_{f_{i}}   = & b_{i}\xi_{1}+\sum_{j=1}^kd_{j,i}\zeta_{j},\ \ \ \ \ \ 1\leq i\leq k,\\
& \varphi(x)_{f_{i}}   = & b_{i}\xi_{1}+b_{i-k}\xi_{2}+\sum_{j=1}^kd_{j,i}\zeta_{j}+\\
&&+a_{1}\sum_{j=1}^kd_{j,i-k}\zeta_{k+j},\ \ k+1\leq i\leq 2k.\\
\end{eqnarray*}
The coordinates of $\Delta(x)$ are

\begin{eqnarray*}
& \Delta(x)_{e_1}   =& \gamma_{11}\xi_{1},\\
& \Delta(x)_{e_i}   =& \sum_{j=1}^{i}\gamma_{i,j}\xi_{j}, \ \ \ \ 2\leq i\leq {n-2k-1},\\
& \Delta(x)_{e_{n-2k}}=& \sum_{j=1}^{n-2k}\gamma_{n-2k,j}\xi_{j}+\sum_{j=1}^{k}\gamma_{n-2k,n-2k+j}\zeta_{j},\\
& \Delta(x)_{f_i}   =& \gamma_{n-2k+i,1}\xi_{1}+\sum_{j=1}^k\gamma_{n-2k+i,n-2k+j}\zeta_{j},\ \ \ \ \ \ 1\leq i\leq k,\\
& \Delta(x)_{f_i}   =& \gamma_{n-2k+i,1}\xi_{1}+\gamma_{n-2k+i,2}\xi_{2}+\sum_{j=1}^{2k}\gamma_{n-2k+i,n-2k+j}
\zeta_{j},\,\, k+1\leq i\leq 2k.\\
\end{eqnarray*}
Comparing the coordinates of $\Delta(x)$ and $D(x),$ we obtain
\begin{equation}\label{for1.4.2}
{\small\left\{
\begin{array}{llllllllllll}
a_{1}\xi_{1} & =  \gamma_{11}\xi_{1} \\
a_{i}\xi_{1}+\sum\limits_{j=1}^{i-2}a^{j}_1a_{i-j}\xi_{j+1}+a_{1}^i\xi_{i}
& =  \sum\limits_{j=1}^{i}\gamma_{i,j}\xi_{j},\,\,\,\,  2\leq i\leq {n-2k-1}, \\
a_{n-2k}\xi_{1}+ \sum\limits_{j=1}^{n-2k-2}a_{n-2k-j}\xi_{j+1}+(n-2k)a_{1}\xi_{n-2k}+
 & \,
\\
+ \sum\limits_{j=1}^kc_{j,n-2k}\zeta_{j}
& =   \sum\limits_{j=1}^{n-2k}\gamma_{n-2k,j}\xi_{j}+\sum\limits_{j=1}^{k}\gamma_{n-2k,n-2k+j}\zeta_{j},\\
b_{i}\xi_{1}+\sum\limits_{j=1}^kd_{j,i}\zeta_{j}
& =  \gamma_{n-2k+i,1}\xi_{1}+\\
\, & +\sum\limits_{j=1}^k\gamma_{n-2k+i,n-2k+j}\zeta_{j},\,\,\,   1\leq i\leq k,\\
b_{i}\xi_{1}+b_{i-k}\xi_{2}+\sum\limits_{j=1}^kd_{j,i}\zeta_{j}+&
\,\\
+ a_{1}\sum\limits_{j=1}^kd_{j,i-k}\zeta_{k+j} & =
\gamma_{n-2k+i, 1}\xi_{1}
+\gamma_{n-2k+i,2}\xi_{2}+\\
\, & +\sum\limits_{j=1}^{2k}\gamma_{n-2k+i,n-2k+j}\zeta_{j},
\,\,\, k+1\leq i\leq 2k.
\end{array} \right.}
\end{equation}

We show the solvability of this system of equations with respect to $a_i,  b_i, c_i$ and  $d_{i,j}.$ For this purpose we consider the following possible cases.

\textbf{Case 1.} Let  $\xi_{1}\neq0$, then  putting
$c_{i}=d_{i,j}=0, \ 1\leq i\leq k,\  1\leq j\leq k$ from \eqref{for1.4.2}
we uniquely determine $a_1, a_2,\ldots, a_{n-2k}, b_1, b_2, \ldots, b_{2k}.$

\textbf{Case 2.} Let  $\xi_{1}=0$ and $\xi_{2}\neq0$, then putting $c_{i}=d_{i,j}=0, \ 1\leq i\leq k,\ 1\leq j\leq
k$ we uniquely determine remaining unknowns $a_1,a_2, \ldots, a_{n-2k}, b_1, b_2, \ldots, b_{k}$.

\textbf{Case 3.} Let $\xi_{1}=\xi_{2}=\dots=\xi_{r-1}=0$ and
$\xi_{r}\neq0, \ 3\leq r \leq n-2k$. Then putting
$$a_2=\ldots=a_{n-2k-m} = b_1 = \ldots = c_1=\ldots =d_{t,j}=0, \ 1\leq t\leq k,\ 1\leq j\leq k.$$
we determine unknowns $a_1,\ a_{n-2k-m+1}, \ i \leq m \leq n-2k$.

\textbf{Case 4.} Let $\xi_{1}=\ldots=\xi_{n-2k}=\zeta_1=\ldots
=\zeta_{r-1}=0$ and $\zeta_{r}\neq0, \ 1\leq r \leq k$. Then setting
$$a_2=\ldots=b_1=\ldots=0,\ c_{i}=0,\, i\neq r,\ d_{j,i}=0, \ j\neq r,$$
we determine $a_1,\ c_r, d_ {r,i}, \ 1 \leq i \leq k$.

 \textbf{Case 5.} Let $\xi_{1}=\ldots =\xi_{n-2k}=\zeta_1=\ldots =\zeta_{k+r-1}=0$ and
 $\zeta_{k+r}\neq0, \ 1\leq r \leq k$. Then setting
 $$a_2=\ldots=b_1=\ldots=c_1=\ldots=0,\ d_{j, i}=0,\ r\neq k+r,$$
we obtain that the unknowns $a_1,\ d_{k+r,i}, k+1\leq i \leq 2k,$ are uniquely determined from \eqref{for1.4.2}.
\end{proof}

In the following theorems we obtain the descriptions of local automorphism of the algebras  $\mu_{2}$ and $\mu_3$.

\begin{thm}\label{sec1.4.3}
Let $\Delta$ be a linear operator on $\mu_{2}$. Then $\Delta$
is a local automorphism if and only if its matrix has the form:
\begin{eqnarray*}
\Delta= \left(\begin{array}{cc}
\Delta_{1,1} & \Delta_{1,2}\\
\Delta_{2,1} & \Delta_{2,2}
\end{array}\right),
\end{eqnarray*}
where
$$
 \Delta_{1,1}=  \sum\limits_{j=1}^{n-2k}\sum\limits_{i=j}^{n-2k}\alpha_{j,i}e_{j,i}, \quad
\Delta_{2,1}= \sum\limits_{i=1}^{2k}\beta_{n-2k+i,1}e_{n-2k+i,1}+\sum\limits_{i=1}^{k}b_{n-2k+i,2}e_{n-2k+i,2},$$
$$
\Delta_{1,2}=  \sum_{i=1}^{k}\gamma_{n-2k,n-2k+i}e_{n-2k,n-2k+i},$$
$$
\Delta_{2,2}=  \begin{pmatrix}
\Delta_{2,2}^{(1)} & 0  \\
\Delta_{2,2}^{(2)} & \Delta_{2,2}^{(3)}
\end{pmatrix},$$
$$
\Delta_{2,2}^{(1)} = \delta_{n-2k+1,n-2k+1}e_{n-2k+1,n-2k+1}+ \sum\limits_{i=n-2k+1}^{n-k}\sum\limits_{j=n-2k+2}^{n-k}\delta_{j,i}e_{j,i},$$
$$
\Delta_{2,2}^{(2)}= \sum\limits_{i=n-k+1}^{n-k}\sum\limits_{j=n-k+1}^{n}\delta_{j,i}e_{j,i},$$
$$
\Delta_{2,2}^{(3)}=\delta_{n-k+1,n-k+1}e_{n-k+1,n-k+1}+\sum\limits_{i=n-k+1}^{n}\sum\limits_{j=n-k+2}^{n}\delta_{j,i}e_{j,i}.$$
\end{thm}
\begin{proof} The proof is similar to the proof of Theorem \ref{sec1.4.2}
\end{proof}

\begin{thm}\label{sec1.4.4}
Let $\Delta$ be a linear operator on $\mu_{3}$. Then $\Delta$
is a local automorphism if and only if its matrix has the form:
\begin{eqnarray*}
\Delta= \left(\begin{array}{cc}
\Delta_{1,1} & \Delta_{1,2}\\
\Delta_{2,1} & \Delta_{2,2}
\end{array}\right),
\end{eqnarray*}
where
$$
 \Delta_{1,1}=  \sum\limits_{j=1}^{n-2k}\sum\limits_{i=j}^{n-2k}\alpha_{j,i}e_{j,i}, \quad \alpha_{i,1}=\alpha_{i,2}, \ \alpha_{2,2}=\alpha_{1,1}+\alpha_{2,1}, \ 3\leq i\leq n-2k-1$$
$$
 \Delta_{2,1}= \sum\limits_{i=1}^{2k}\beta_{n-2k+i,1}e_{n-2k+i,1}+\sum\limits_{i=1}^{k}b_{n-2k+i,2}e_{n-2k+i,2}+\sum\limits_{i=1}^{k}b_{n-2k+i,3}e_{n-2k+i,3},$$
$$
\Delta_{1,2}=  \sum_{i=1}^{k}\gamma_{n-2k,n-2k+i}e_{n-2k,n-2k+i},$$
$$
\Delta_{2,2}=  \begin{pmatrix}
\Delta_{2,2}^{(1)} & 0  \\
\Delta_{2,2}^{(2)} & \Delta_{2,2}^{(3)}
\end{pmatrix},$$
$$
\Delta_{2,2}^{(1)} =  \sum\limits_{i=n-2k+1}^{n-k}\sum\limits_{j=n-2k+1}^{n-k}\delta_{j,i}e_{j,i},\quad
\Delta_{2,2}^{(2)}= \sum\limits_{i=n-k+1}^{n-k}\sum\limits_{j=n-k+1}^{n}\delta_{j,i}e_{j,i},$$
$$
 \Delta_{2,2}^{(3)}=  \sum\limits_{i=n-k+1}^{n}\sum\limits_{j=n-k+1}^{n}\delta_{j,i}e_{j,i}
$$
\end{thm}

\begin{proof} The proof is similar to the proof of Theorem \ref{sec1.4.2}
\end{proof}

\begin{exm}\label{Locaut1}
 Leibniz algebras $\mu_1$ (see Proposition \ref{prop1}), admit a local automorphisms which is not a automorphisms.
\end{exm}

\begin{proof} Let us consider the linear operator $\Phi$ on $\mathfrak{L}$, such that
$$\Phi\left(x\right)=x+x_2e_{n-2k},\ \ x=\sum\limits_{i=1}^{n-2k}x_ie_i+\sum\limits_{i=1}^{2k}x_if_{i}$$

By Lemma \ref{prop1}, it is not difficult to see that $\Phi$ is not a automorphism.
We show that, $\Phi$ is a local automorphism on $\mu_1.$

Consider the automorphism $\varphi_1$ and $\varphi_2$ on the algebras $\mu_1,$ defined as:
\begin{equation*}\begin{split}
\varphi_1\left(x\right)&=x+x_1e_{n-2k-1}+x_2e_{n-2k},\\
\varphi_2\left(x\right)&=x+\beta x_1e_{n-2k},\ \ \ x=\sum\limits_{i=1}^{n-2k}x_ie_i+\sum\limits_{i=1}^{2k}x_if_{i}.
\end{split}
\end{equation*}

Now, for any $x=\sum\limits_{i=1}^{n-2k}x_ie_i+\sum\limits_{i=1}^{2k}x_if_{i},$ we shall find a automorphism $\varphi,$ such that $\Phi(x) = \varphi(x).$

If  $x_1=0,$ then
$$\varphi_1(x)=x+x_2e_{n-2k}=\Phi(x).$$

If $x_1\neq 0,$ then set $\beta=\frac{x_2}{x_1},$ we obtain that
\begin{equation*}\begin{split}
 \varphi_2(x)=x+\beta x_1e_n=x+\frac{x_2}{x_1}x_1e_{n-2k}=x+x_2e_{n-2k}=\Phi(x)
\end{split}\end{equation*}

Hence, $\Phi$ is a local automorphism.
\end{proof}

\begin{rem}\label{remsec 1.4.2}
The dimensions of the space of local automorphism of algebras $\mu_{1}, \mu_{2}$ and $\mu_{3}$ are
\begin{eqnarray*}
& \dim LocDer(\mu_{1})= & \frac{n^2+10k^2-4kn+n+6k}{2}, \\
& \dim LocDer(\mu_{2})= & \frac{n^2+10k^2-4kn+n+2k+4}{2} ,  \\
& \dim LocDer(\mu_{3})= & \frac{n^2+10k^2-4kn-n+12k+4}{2},
\end{eqnarray*}
where $k\in \mathbb{N}$ and $n\geq 2k+4.$
\end{rem}

Remarks \ref{remsec 1.4.1} and \ref{remsec 1.4.2} show that the dimensions of the spaces of all local automorphism of the algebras $\mu_i,$ $i = 1,2,3,$ are strictly greater than the dimensions of the space of all automorphism of $\mu_i.$ Therefore, we have the following result.

\begin{cor}
The algebras $\mu_{1}, \mu_{2}$ and $\mu_{3}$ admit local
automorphism which are not automorphism.
\end{cor}

\end{document}